\documentclass[12pt, reqno]{amsart}

\numberwithin{equation}{section}

\makeatletter
\renewcommand{\@biblabel}[1]{#1\hfill \hspace{0 cm}}
\makeatother

\usepackage{amsmath, amssymb, amsthm, amsfonts, cite, dsfont, enumerate, epsfig, float, geometry, doi,
infwarerr, mathrsfs, mathtools, mathrsfs, mathtools, microtype, relsize, stmaryrd, tabularx, txfonts,
nicefrac, subfig}
\usepackage[normalem]{ulem}

\newtheorem{theorem}{Theorem}[section]
\newtheorem{lemma}[theorem]{Lemma}

\newtheorem{corollary}[theorem]{Corollary}

\theoremstyle{definition}
\newtheorem{definition}[theorem]{Definition}

\newtheorem{remark}[theorem]{Remark}
\numberwithin{equation}{section}

\newcommand{\naturals}{\ensuremath{\mathbb{N}}}

\newcommand{\Reals}{\ensuremath{\mathbb{R}}}

\newcommand{\set}{\ensuremath{\mathcal}}
\newcommand{\Field}{\ensuremath{\mathbb{F}}}

\newcommand{\OneTo}[1]{[#1]}
\newcommand{\card}[1]{|#1|}

\newcommand{\Tr}{\mathrm{Tr}}

\DeclareMathOperator{\Vertex}{\mathsf{V}}
\DeclareMathOperator{\Edge}{\mathsf{E}}
\DeclareMathOperator{\Adjacency}{\mathbf{A}}

\DeclareMathOperator{\AllOne}{\mathbf{J}}
\DeclareMathOperator{\Identity}{\mathbf{I}}

\DeclareMathOperator{\Zero}{\mathbf{0}}

\DeclareMathOperator{\Clique}{\omega}
\DeclareMathOperator{\Chromatic}{\chi}

\newcommand{\Gr}[1]{\mathsf{#1}}                          
\newcommand{\CGr}[1]{\overline{\mathsf{#1}}}              
\newcommand{\V}[1]{\Vertex(#1)}                           
\newcommand{\E}[1]{\Edge(#1)}                             
\newcommand{\Ei}{\mathbf{E}}                              
\newcommand{\A}{\Adjacency}                               
\newcommand{\I}[1]{\Identity_{#1}}                        

\newcommand{\indnum}[1]{\alpha(#1)}                       

\newcommand{\clnum}[1]{\Clique(#1)}                       

\newcommand{\chrnum}[1]{\Chromatic(#1)}                   

\newcommand{\fchrnum}[1]{\Chromatic_{\mathrm{f}}(#1)}     

\newcommand{\Eigval}[2]{\lambda_{#1}(#2)}                 

\newcommand{\wt}[1]{w_H(\mathbf{#1})} 
\newcommand{\inner}[2]{\left\langle #1 \, , \, #2 \right\rangle} 
\newcommand{\boseMesner}{\mathcal{B}}
\newcommand{\parity}{\mathrm{par}}
\newcommand{\Dh}{\mathrm{d}_H}
\newcommand{\DG}{\mathrm{d}_{\Gr{G}}}

\newcommand{\Cayley}[2]{\mathrm{Cay}\left({#1},{#2}\right)} 
\DeclareMathOperator{\Aut}{\mathrm{Aut}} 

\newcommand{\Gilbert}[1]{\mathcal{G}_{#1}} 

\DeclareMathOperator{\Neighbors}{\set{N}}

\newcommand{\trace}[1]{\text{Tr}{(#1)}}

\MHInternalSyntaxOn
\renewcommand{\dcases}
{
\MT_start_cases:nnnn
{\quad}
{$\m@th\displaystyle##$\hfil}
{$\m@th\displaystyle##$\hfil}
{\lbrace}
}
\MHInternalSyntaxOff

\makeatletter\c@MaxMatrixCols=15\makeatother

\title{On the transitivity of Gilbert graphs and their complements}

\author[N. Krupnik]{Noam Krupnik}
\address{N. Krupnik, Faculty of Computer Science,
Technion--Israel Institute of Technology,
Haifa 3200003, Israel. E-mail: noam.krupnik@campus.technion.ac.il}

\author[I. Sason]{Igal Sason}
\address{I. Sason, The Viterbi Faculty of Electrical and Computer Engineering,
Technion--Israel Institute of Technology,
Haifa 3200003, Israel;
Faculty of Mathematics,
Technion--Israel Institute of Technology,
Haifa 3200003, Israel. E-mail: eeigal@technion.ac.il}

\author[A. Berman]{Abraham Berman}
\address{A. Berman, Faculty of Mathematics,
Technion--Israel Institute of Technology,
Haifa 3200003, Israel. E-mail: berman@technion.ac.il}

\begin{document}

\date{\today}

\vspace*{0.8cm}
\begin{abstract}
The Gilbert graph $\Gilbert{q,n,d}$, which arises naturally in graph theory and coding theory,
is the regular graph on $\mathbb{F}_q^n$ in which two vertices are adjacent if their Hamming distance
is less than $d$, and it is vertex-transitive. We classify all parameters $(q,n,d)$ for which
$\Gilbert{q,n,d}$ is edge-transitive or distance-transitive, and separately classify all parameters
for which its complement has these properties. We prove that $\Gilbert{q,n,d}$ is edge-transitive
if and only if it is distance-transitive, and that this occurs precisely when $d=2$, $(q,d)=(2,3)$,
or $(q,d)=(2,n)$. For the complement graphs, we determine all parameters yielding edge- or
distance-transitivity using spectral methods based on Krawtchouk polynomials and the structure of
the Hamming association scheme. In contrast to the Gilbert graphs, where the parameter sets corresponding
to edge- and distance-transitivity coincide, we show that for their complements the set of parameters
yielding distance-transitivity is strictly contained in the set yielding edge-transitivity. As an
application, we compute the exact values of the Lov\'{a}sz $\vartheta$-function of Gilbert graphs,
as well as of their complements, in all cases where either one of them is edge-transitive.
\end{abstract}

\maketitle
\thispagestyle{empty}


\vspace*{0.2cm}

\maketitle

\section{Introduction}
\label{section: Introduction}

Let $q\in \naturals$ be a prime power, and let $n,d \in \naturals$ with $2 \le d \le n$.
The term \emph{Gilbert graph} is used in this paper to denote the graph with vertex set $\mathbb{F}_q^n$
in which two distinct vertices are adjacent if their Hamming distance is less than $d$.
This graph, denoted by $\Gilbert{q,n,d}$, arises in information theory in connection with both the classical
Gilbert--Varshamov bound and the Shannon capacity of graphs (see \cite{JiangVardy2004, Sason2025, McElieceRR1978, Tolhuizen1997, YeZLWYMa2023}).
The term \emph{Gilbert graph} is also used in the literature to describe several different objects: in probabilistic
combinatorics, it often refers to the random graph $G(n,p)$ from the Erd\H{o}s--R\'{e}nyi ensemble (independently introduced by Gilbert),
while in geometric probability, it may denote the random geometric graph obtained by connecting points whose Euclidean distance is at most
a prescribed threshold. To avoid ambiguity, throughout this paper, the term \emph{Gilbert graph} will always refer to the distance-threshold
graph on $\mathbb{F}_q^n$, as defined above. In the literature, the term \emph{generalized-Hamming graph} is also used to refer to this graph 
or to closely related variants \cite{AffifChaoucheB2006, LiZW2023, QianWX2022, WangCLL2026}.

In \cite{Sason2025}, the Gilbert graph $\Gilbert{2,5,3}$ was used to show that Schrijver's $\vartheta$-function need not be an upper bound
on the Shannon capacity of graphs. Relying on \cite{Sason2024}, the selection of that graph as a potential example was justified
by the fact that the graph and its complement are both vertex-transitive, the graph is edge-transitive, whereas its complement
is not (had the graph and its complement both been vertex- and edge-transitive, it could not serve as an example \cite{Sason2024}).
This observation motivates the present paper, whose goal is to classify the parameters $q, n, d$ for which either the Gilbert graph or its complement
is edge-transitive. We show that in all cases where the Gilbert graph is edge-transitive, it is even distance-transitive (recall that
distance-transitivity implies edge-transitivity, but the opposite is not necessarily true). The paper also classifies the parameters $q, n, d$
for which the complement of the Gilbert graph, denoted by $\overline{\Gilbert{q,n,d}}$, is edge-transitive, and those for which
it is distance-transitive. The proof of the latter classification relies on spectral methods in graph theory and the theory of association
schemes. To the best of our knowledge, a complete classification of the parameters $q, n, d$ for which Gilbert graphs or their complements are
edge-transitive or distance-transitive has not previously appeared in the literature.

The main results in this paper are as follows.
\begin{theorem}
\label{thm: main result}
Let $q \in \naturals$ be a prime power and let $n,d \in \naturals$ with $2 \leq d \le n$. Set 
$\Gr{G} = \Gilbert{q,n,d}$. Then, the following statements are equivalent:
\begin{enumerate}
\item $\Gr{G}$ is edge-transitive.
\item $\Gr{G}$ is distance-transitive.
\item Either $d = 2$, or $(q,d) = (2,3)$, or $(q,d) = (2,n)$.
\end{enumerate}
\end{theorem}

\begin{theorem}
\label{thm: main result complement}
Let $q \in \naturals$ be a prime power and let $n,d \in \naturals$ with $2 \leq d \le n$. Set $\Gr{G} = \overline{\Gilbert{q,n,d}}$. Then,
\begin{enumerate}
\item $\Gr{G}$ is edge-transitive if and only if either $(q,d)=(2,n-1)$ for an even $n$, or $d = n$.
\item $\Gr{G}$ is distance-transitive if and only if one of the following two conditions holds:
\begin{enumerate}
\item $(q,d)=(2,n-1)$ for an even $n$,
\item $d = n$ and either $q=2$ or $n=2$.
\end{enumerate}
\end{enumerate}
\end{theorem}

\begin{corollary}  \label{corollary: Gilbert graphs and complements}
Let $\Gr{G}$ be a Gilbert graph such that $\Gr{G}$ and $\CGr{G}$ are both edge-transitive. Then, either $(q,d) = (2,n)$, $(q,d,n) = (2,3,4)$, or $(d,n)=(2,2)$.
\end{corollary}

The paper alternates between background material and new results.
The even-numbered sections provide necessary background on graph transitivity, association schemes, Krawtchouk polynomials, and the Lov\'{a}sz
$\vartheta$-function, while the odd-numbered sections present the proofs of our results and their applications.
An exception is Section~\ref{subsection: Spectrum of the complement of the Gilbert graph}, where we derive spectral properties
of the complement of Gilbert graphs. Since these results appear to be new, Section~\ref{subsection: Spectrum of the complement of the Gilbert graph}
contains a complete proof instead of referring to external sources.

In regard to the odd-numbered sections, Theorem~\ref{thm: main result} is proved in Section~\ref{sec: transitive gilbert graphs}, and
Theorem~\ref{thm: main result complement} is proved in Section~\ref{sec: transitive complement gilbert graphs}. The results are then
applied in Section~\ref{sec: applications} to compute the Lov\'{a}sz $\vartheta$-function of edge-transitive Gilbert graphs and their complements.

\section{Graph transitivity}
\label{section: transitivity}

All graphs in this paper are finite, simple, undirected, and connected. We denote by $\Gr{G}$ a graph with vertex set $\V{\Gr{G}}$ and edge
set $\E{\Gr{G}}$. For every vertex $v \in \V{\Gr{G}}$, we denote by $\Neighbors(v)$ the neighborhood of $v$, i.e.,
the set of vertices adjacent to $v$. The \emph{distance} between two vertices $u,v \in \V{\Gr{G}}$ is the length of the
shortest path between them and is denoted by $d_{\Gr{G}}(u,v)$. A \emph{graph automorphism} is a permutation of the vertex set that preserves
edges and non-edges. The set of all graph automorphisms of $\Gr{G}$ forms an \emph{automorphism group}, denoted by $\Aut(\Gr{G})$.
The following types of graph symmetries are considered:
\begin{enumerate}
\item $\Gr{G}$ is \emph{vertex-transitive} if, for every two vertices $u,v \in \V{\Gr{G}}$, there exists an automorphism of $\Gr{G}$ that maps $u$ to $v$.
\item $\Gr{G}$ is \emph{edge-transitive} if, for every pair of edges $\{u_1,v_1\}, \{u_2,v_2\} \in \E{\Gr{G}}$, there exists an automorphism of $\Gr{G}$
that maps $\{u_1, v_1\}$ to $\{u_2, v_2\}$.
\item $\Gr{G}$ is \emph{distance-transitive} if, for every four vertices $u_1, u_2, v_1, v_2 \in \V{\Gr{G}}$ satisfying the equality
$d_{\Gr{G}}(v_1,v_2) = d_{\Gr{G}}(u_1,u_2)$, there exists an automorphism of $\Gr{G}$ that maps $v_1$ to $u_1$ and $v_2$ to $u_2$.
\end{enumerate}
If a graph is distance-transitive, then it is also edge-transitive and vertex-transitive.
A Cayley graph of an abelian group is vertex-transitive \cite[Theorem 3.1.2]{GodsilRoyle2001}, and
the Gilbert graph is a Cayley graph of the additive group $\Field_q^n$. In particular,
\begin{align}
&\Gilbert{q,n,d} = \Cayley{\Field_q^n}{\mathcal{S}_{n,d}} \\
&\mathcal{S}_{n,d} := \left\{\mathbf{x} \in \Field_q^n \, : \, 1 \leq \wt{\mathbf{x}} \leq d-1\right\},
\end{align}
where $\wt{\mathbf{x}}$ is the Hamming weight of the vector $\mathbf{x}$, i.e., the number of nonzero coordinates of $\mathbf{x}$, and
$\mathcal{S}_{n,d}$ is called a {\em connection set}. Similarly,
the complement graph is also a Cayley graph:
\begin{align}
\overline{\Gilbert{q,n,d}} = \Cayley{\Field_q^n}{\Field_q^n \setminus (\mathcal{S}_{n,d} \cup \{\mathbf{0}\})}.
\end{align}
Thus, both the Gilbert graph and its complement are vertex-transitive, but they are not necessarily edge-transitive
(hence, they need not be distance-transitive).

\section{Transitive Gilbert graphs}
\label{sec: transitive gilbert graphs}

This section proves Theorem~\ref{thm: main result}, which fully characterizes the parameters
for which Gilbert graphs are edge-transitive or distance-transitive, and establishes that a
Gilbert graph is edge-transitive if and only if it is distance-transitive.

The Gilbert graph $\Gilbert{q,n,2}$ is also called the Hamming graph, and it
is known to be distance-transitive \cite[Table~6.1]{BrouwerCN1989}. A proof that $\Gilbert{2,n,3}$
is distance-transitive is presented in \cite{Mirafzal2021}. The following lemma completes the proof
of one direction of Theorem~\ref{thm: main result}.

\begin{lemma}
\label{lem: gilbert (2,n,n) is distance-transitive}
The Gilbert graph $\Gilbert{2,n,n}$ is distance-transitive.
\end{lemma}

\begin{proof}
Every Gilbert graph is vertex-transitive since it is a Cayley graph.
Consequently, it is also distance-transitive if and only if for every two vertices $\mathbf{u},\mathbf{v} \in \Field_q^n \setminus \{\mathbf{0}\}$, such that $\DG(\mathbf{0},\mathbf{u}) = \DG(\mathbf{0},\mathbf{v})$, there exists a graph automorphism that maps $\{\mathbf{0},\mathbf{u}\}$ to $\{\mathbf{0},\mathbf{v}\}$.
Let $\Gr{G} = \Gilbert{2,n,n}$. For every vertex $\mathbf{v}\in \Field_2^n$, the only vertex that is not adjacent to $\mathbf{v}$ is its complement $\overline{\mathbf{v}} \triangleq \mathbf{v} \oplus \mathbf{1}_n$, where $\oplus$ is the componentwise modulo-2 sum.
Let $\mathbf{u}$ and $\mathbf{v}$ be two vertices such that $\DG(\mathbf{0},\mathbf{v}) = \DG(\mathbf{0},\mathbf{u})$. Consider the mapping
$\varphi \colon \Field_q^n \to \Field_q^n$ given by
\begin{align}
& \varphi(\mathbf{x}) =
\begin{cases}
\mathbf{u} & \text{if } \mathbf{x} = \mathbf{v}, \\
\mathbf{\overline{u}} & \text{if } \mathbf{x} = \overline{\mathbf{v}}, \\
\mathbf{v} & \text{if } \mathbf{x} = \mathbf{u}, \\
\overline{\mathbf{v}} & \text{if } \mathbf{x} = \overline{\mathbf{u}}, \\
\mathbf{x} & \text{otherwise}.
\end{cases}
\end{align}
If $\DG(\mathbf{0},\mathbf{v}) = \DG(\mathbf{0},\mathbf{u}) = 2$, then $\mathbf{v} = \mathbf{u} = \mathbf{1}_n$ and $\varphi$ is
the identity mapping. Otherwise, we have $\DG(\mathbf{0},\mathbf{v}) = \DG(\mathbf{0},\mathbf{u}) = 1$.
In the latter case, $\varphi$ permutes between $\mathbf{v}$ and $\mathbf{u}$, and also between their complements, while keeping
all other vertices unchanged. Hence, $\varphi$ is a graph automorphism that maps $\{\mathbf{0},\mathbf{v}\}$ to $\{\mathbf{0},\mathbf{u}\}$.
\end{proof}

The next lemma proves that every Gilbert graph not listed in Theorem~\ref{thm: main result} is not edge-transitive. The proof identifies
an invariant of edges in edge-transitive graphs and shows that this invariant takes different values for two edges of the graph.

\begin{lemma}
\label{lem: gilbert only if direction}
Let $q,n,d \in \naturals$ be such that $(q,n,d)$ does not satisfy any of the conditions in Item~3 of Theorem~\ref{thm: main result}.
Then, the Gilbert graph $\Gilbert{q,n,d}$ is not edge-transitive.
\end{lemma}

\begin{proof}
Let $\Gr{G} = \Gilbert{q,n,d}$ be a Gilbert graph such that $(q,n,d)$ does not satisfy any of the conditions in Item~3 of Theorem~\ref{thm: main result},
and suppose to the contrary that $\Gr{G}$ is edge-transitive. Define $F \colon \Neighbors(\mathbf{0}) \to \naturals$ to be the number
of neighbors of a vertex $\mathbf{v} \in \Neighbors(\mathbf{0})$ that are not neighbors of $\mathbf{0} \in \Field_q^n$, i.e.,
\begin{align}
F(\mathbf{v}) = |\Neighbors(\mathbf{v}) \setminus \Neighbors(\mathbf{0})|.
\end{align}
Under the assumption that $\Gr{G}$ is edge-transitive, for every pair of vertices $\mathbf{u}, \mathbf{v} \in \Neighbors(\mathbf{0})$, there
exists an automorphism $\varphi$ of $\Gr{G}$ such that $\varphi(\{\mathbf{0},\mathbf{u}\}) = \{\mathbf{0},\mathbf{v}\}$.
Hence either $\varphi(\mathbf{0}) = \mathbf{0}$ and $\varphi(\mathbf{u}) = \mathbf{v}$,
or $\varphi(\mathbf{0}) = \mathbf{v}$ and $\varphi(\mathbf{u}) = \mathbf{0}$.
In the latter case, since $\Gr{G}$ is a Cayley graph with a symmetric connection set,
both the translation $\tau_{-\mathbf{v}} \colon \mathbf{x} \mapsto \mathbf{x}-\mathbf{v}$
and the negation map $\iota \colon \mathbf{x} \mapsto -\mathbf{x}$ are automorphisms.
Therefore, the composition
\[
\psi := \iota \circ \tau_{-\mathbf{v}} \circ \varphi
\]
is an automorphism of $\Gr{G}$ satisfying $\psi(\mathbf{0}) = \mathbf{0}$ and
$\psi(\mathbf{u}) = \mathbf{v}$. Consequently, for every pair
$\mathbf{u}, \mathbf{v} \in \Neighbors(\mathbf{0})$, there exists an automorphism $\psi$
fixing $\mathbf{0}$ and mapping $\mathbf{u}$ to $\mathbf{v}$.
Graph automorphisms preserve adjacency, and hence
\begin{align*}
F(\mathbf{u}) &= |\Neighbors(\mathbf{u}) \setminus \Neighbors(\mathbf{0})| \\
&= |\Neighbors(\psi(\mathbf{u})) \setminus \Neighbors(\psi(\mathbf{0}))| \\
&= |\Neighbors(\mathbf{v}) \setminus \Neighbors(\mathbf{0})| \\
&= F(\mathbf{v}).
\end{align*}
Hence the value $F(\mathbf{v})$ is identical for all vertices $\mathbf{v} \in \Neighbors(\mathbf{0})$.
In particular, consider the two vertices
\[
\mathbf{u}_1 = (1,0,0,\ldots,0), \qquad \mathbf{u}_2 = (1,1,0,\ldots,0).
\]
Since the third condition in Theorem~\ref{thm: main result} is not satisfied, it follows that $d \ge 3$.
Hence, $\mathbf{u}_1, \mathbf{u}_2 \in \Neighbors(\mathbf{0})$, and therefore $F(\mathbf{u}_1) = F(\mathbf{u}_2)$.

We use combinatorial arguments to derive closed-form expressions for $F(\mathbf{u}_1)$ and $F(\mathbf{u}_2)$. We adopt
the convention for binomial coefficients that $\binom{n}{k} = 0$ if $k > n$ or $k < 0$.

The neighbors of $\mathbf{u}_1$ that are not neighbors of $\mathbf{0}$ are precisely the vertices in $\mathbb{F}_q^n$
whose first coordinate equals $1$ and whose Hamming weight is $d$. Hence, the number of such vertices is
\begin{align}
\label{eq: F(u1)}
F(\mathbf{u}_1) = \binom{n-1}{d-1} \, (q-1)^{d-1}.
\end{align}
The set of neighbors of $\mathbf{u}_2$ that are not neighbors of $\mathbf{0}$ is the disjoint union of the following sets:
\begin{itemize}
\item The neighbors of $\mathbf{u}_2$ that are not neighbors of $\mathbf{0}$ and have Hamming weight $d$. Such vertices cannot begin with the
subsequences $(0,0)$, $(0,1)$, or $(1,0)$, since in each of these cases their Hamming weight would be less than $d$, and hence they would be
adjacent to $\mathbf{0}$. Therefore, they must begin with $(1,1)$. The number of such vertices is $\binom{n-2}{d-2} \, (q-1)^{d-2}$.
\item The neighbors of $\mathbf{u}_2$ that are not neighbors of $\mathbf{0}$ and have Hamming weight $d+1$. By the same reasoning, each of
these vertices starts with $(1,1)$, and their number is $\binom{n-2}{d-1} \, (q-1)^{d-1}$.
\end{itemize}
Overall,
\begin{align}
\label{eq: F(u2)}
F(\mathbf{u}_2) = \binom{n-2}{d-2} \, (q-1)^{d-2} + \binom{n-2}{d-1} \, (q-1)^{d-1}.
\end{align}
Consequently, by \eqref{eq: F(u1)} and \eqref{eq: F(u2)}, the equality $F(\mathbf{u}_2) - F(\mathbf{u}_1) = 0$ yields
\begin{align}
&\binom{n-2}{d-2} \, (q-1)^{d-2} + \binom{n-2}{d-1} \, (q-1)^{d-1} - \binom{n-1}{d-1} \, (q-1)^{d-1} = 0, \\[0.1cm]
\iff \quad & \binom{n-2}{d-2} + \binom{n-2}{d-1} \, (q-1) - \binom{n-1}{d-1} \, (q-1) = 0, \\[0.1cm]
\iff \quad & \frac{1}{n-d} + \frac{q-1}{d-1} - \frac{(n-1)(q-1)}{(n-d)(d-1)} = 0, \\[0.1cm]
\iff \quad & (d-1) + (n-d)(q-1) - (n-1)(q-1) = 0, \\[0.1cm]
\iff \quad & (q-2)(1-d) = 0. \label{eq: 17.02.26}
\end{align}
By \eqref{eq: 17.02.26}, if $q \ne 2$, we obtain a contradiction to the assumption that $\Gr{G}$ is edge-transitive (recall that $d \geq 3$).

Otherwise, if none of the conditions in Item~3 of Theorem~\ref{thm: main result} are satisfied, then $q = 2$ and $4 \leq d \leq n-1$.
In the latter case,
\[
\mathbf{u}_3 \triangleq (1,1,1,0,\ldots,0) \in \Neighbors(\mathbf{0}).
\]
Let $q = 2$ and $4 \leq d \leq n-1$.
We calculate $F(\mathbf{u}_3)$ and compare it to $F(\mathbf{u}_1)$, since, by the above, these two values are identical as
$\mathbf{u}_1, \mathbf{u}_3 \in \Neighbors(\mathbf{0})$.
The set of neighbors of $\mathbf{u}_3$ that are not neighbors of $\mathbf{0}$ is the disjoint union of the following sets:
\begin{itemize}
\item The neighbors of $\mathbf{u}_3$ with Hamming weight $d$. These vertices can differ from $\mathbf{u}_3$ in at most one
of the first three coordinates; hence, the Hamming weight of their first three coordinates must be $2$ or $3$. We consider
these two cases separately.
\begin{enumerate}
\item If this Hamming weight is $2$, then the first three coordinates of these vertices are either $(0,1,1)$, $(1,0,1)$, or
$(1,1,0)$, and in each of these three cases, there are $\binom{n-3}{d-2}$ possibilities for the remaining coordinates.
\item If the first three coordinates are $(1,1,1)$, then there are $\binom{n-3}{d-1}$ possibilities for the remaining coordinates.
\end{enumerate}
Therefore, the total number of neighbors of $\mathbf{u}_3$ with Hamming weight $d$ is
\[
3 \, \binom{n-3}{d-2} + \binom{n-3}{d-3}.
\]
\item The neighbors of $\mathbf{u}_3$ with Hamming weight $d+1$. Such vertices must have their first three coordinates equal to
$(1,1,1)$, and there are $\binom{n-3}{d-2}$ such vertices.
\item The neighbors of $\mathbf{u}_3$ with Hamming weight $d+2$. Again, the first three coordinates must be $(1,1,1)$, and there
are $\binom{n-3}{d-1}$ such vertices.
\end{itemize}
Overall,
\begin{align}  \label{eq2: 17.02.26}
F(\mathbf{u}_3) &= 4 \, \binom{n-3}{d-2} + \binom{n-3}{d-3} + \binom{n-3}{d-1}.
\end{align}
By invoking Pascal's identity
\[
\binom{a}{b} = \binom{a-1}{b-1} + \binom{a-1}{b}, \quad a,b \in \naturals,
\]
three times, equality \eqref{eq2: 17.02.26} can be rewritten as
\begin{align}
F(\mathbf{u}_3) &= 3 \, \binom{n-3}{d-2} + \Biggl[ \binom{n-3}{d-2} + \binom{n-3}{d-3} \Biggr] + \binom{n-3}{d-1} \nonumber \\
&= 3 \, \binom{n-3}{d-2} + \binom{n-2}{d-2} + \binom{n-3}{d-1} \nonumber \\
&= 2 \, \binom{n-3}{d-2} + \Biggl[ \binom{n-3}{d-2} + \binom{n-3}{d-1} \Biggr] + \binom{n-2}{d-2} \nonumber \\
&= 2 \, \binom{n-3}{d-2} + \binom{n-2}{d-1} + \binom{n-2}{d-2} \nonumber \\
\label{eq: F(u3)}
&= 2 \, \binom{n-3}{d-2} + \binom{n-1}{d-1}.
\end{align}
Consequently, by \eqref{eq: F(u1)} (for $q=2$) and \eqref{eq: F(u3)}, the equality $F(\mathbf{u}_3) - F(\mathbf{u}_1) = 0$ yields
\begin{align}
\label{eq1: 09.03.26}
F(\mathbf{u}_3) - F(\mathbf{u}_1) = 2 \, \binom{n-3}{d-2} = 0.
\end{align}
Since $4 \leq d \leq n-1$ in the considered case, equality \eqref{eq1: 09.03.26} holds if and only if $d-2 > n-3$, i.e., $d > n-1$.
This leads to a contradiction. Therefore, in this case as well, we obtain a contradiction to the assumption that $\Gr{G}$ is edge-transitive.
\end{proof}

\section{Association Schemes}
\label{sec: association schemes}

The proof of Theorem~\ref{thm: main result complement}, which appears in the next section (Section~\ref{sec: transitive complement gilbert graphs}),
relies on the theory of association schemes. This section begins by briefly recalling the necessary definitions and properties of association schemes.

\subsection{Association schemes and their matrix representation}
\label{subsection: Association schemes and their matrix representation}

\begin{definition}[association scheme]
\label{def: association scheme}
Let $\mathcal{X}$ be a finite set and let $m \in \naturals$.
An \emph{$(m+1)$-class association scheme} on $\mathcal{X}$ is a partition
$\mathcal{R} = \{R_0, R_1, \ldots, R_m\}$ of $\mathcal{X} \times \mathcal{X}$
(i.e., $\bigcup_{i=0}^{m} R_i = \mathcal{X} \times \mathcal{X}$ and
$R_i \cap R_j = \varnothing$ for all $i \neq j$) that satisfies the following properties:
\begin{itemize}
\item $R_0 = \{(x,x): \, x \in \mathcal{X}\}$ is the identity relation.
\item For every $0 \le i \le m$, the relation $R_i^{\star} \triangleq \{(y,x): \, (x,y) \in R_i\}$ also belongs to $\mathcal{R}$
(i.e., $R_i^{\star} = R_j$ for some $j$).
In the special case where $R_i^{\star} = R_i$ for every $0 \le i \le m$, the association scheme is called \emph{symmetric}.
\item For every $0 \le i,j,k \le m$, there exist integers $p_{ij}^{k}$ such that for every $(x,y) \in R_k$,
\[
\Bigl| \, \bigl\{z \in \mathcal{X}: \, (x,z) \in R_i \text{ and } (z,y) \in R_j\bigr\} \, \Bigr| = p_{ij}^{k},
\]
i.e., the above quantity depends only on $i,j,k$ and not on the particular pair $(x,y) \in R_k$.
The integers $\{p_{ij}^{k}\}$ are called the \emph{intersection numbers} of the association scheme.
\end{itemize}
\end{definition}

For every $i \in \{0, 1, \ldots, m\}$, the \emph{adjacency matrix representation} of the relation $R_i$ is the matrix
$\A_i \in \{0,1\}^{|\mathcal{X}| \times |\mathcal{X}|}$, where
\begin{align}
\A_i(x,y) =
\begin{cases}
1 & \text{if } (x,y) \in R_i, \\
0 & \text{otherwise.}
\end{cases}
\end{align}

One can represent the relations of an association scheme by adjacency matrices. Accordingly, an association scheme
$\mathcal{R} = \{R_0, R_1, \ldots, R_m\}$ can be represented by the set of matrices $\{\A_0, \A_1, \ldots, \A_m\}$. In this
representation, the conditions of Definition~\ref{def: association scheme} are rewritten as follows:
\begin{itemize}
\item $\sum_{i=0}^{m} \A_i = \AllOne$, where $\AllOne$ denotes the all-ones matrix of size $\card{\mathcal{X}} \times \card{\mathcal{X}}$.
\item $\A_0 = \Identity$ is the identity matrix of size $\card{\mathcal{X}} \times \card{\mathcal{X}}$.
\item For every $0 \le i \le m$, the equality $\A_i^T = \A_{j}$ holds for some $0 \le j \le m$.
If the association scheme is symmetric then $\A_i^T = \A_i$ for every $0 \le i \le m$.
\item For every $0 \le i,j \le m$,
\begin{align}
\label{eq: 21.03.26}
\A_i \A_j = \sum_{k=0}^{m} p_{ij}^k \A_k = \A_j \A_i.
\end{align}
\end{itemize}
Since the matrices $\A_0, \A_1, \ldots, \A_m$ commute pairwise, they admit a common basis of eigenvectors and can therefore be simultaneously
diagonalized. The matrices $\A_0, \A_1, \ldots, \A_m$ span a commutative algebra called the \emph{Bose--Mesner algebra}.
\begin{definition}[Bose--Mesner algebra]
Let $\mathcal{R}$ be an association scheme on a finite set $\mathcal{X}$ with adjacency matrices $\A_0,\dots,\A_m$.
The \emph{Bose--Mesner algebra} of $\mathcal{R}$, denoted by $\boseMesner(\mathcal{R})$, is the subalgebra of
$\mathbb{R}^{|\mathcal{X}|\times|\mathcal{X}|}$ consisting of all real linear combinations of these matrices:
\[
\boseMesner(\mathcal{R})=\mathrm{span}_{\mathbb{R}}\{\A_0,\dots,\A_m\}.
\]
Since the products $\A_i\A_j$ are again linear combinations of $\A_0,\dots,\A_m$ (see \eqref{eq: 21.03.26}), this span is closed under matrix multiplication.
\end{definition}

\subsection{Dual basis and eigenmatrices}
\label{subsection: Dual basis and eigenmatrices}
Apart from the basis formed by the adjacency matrices, the Bose--Mesner algebra
has another distinguished basis called the \emph{dual basis}, consisting of
orthogonal projections (the primitive idempotents). The following theorem
states this fact.

\begin{theorem}
\label{thm: dual basis of association scheme}
Let $\mathcal{R}$ be an association scheme on a set $\mathcal{X}$ with
adjacency matrices $\{\A_0,\A_1,\ldots,\A_m\}$. Then there exists a basis of
$\boseMesner(\mathcal{R})$ consisting of orthogonal projections
$\Ei_0,\Ei_1,\ldots,\Ei_m$. In particular, there exist scalars $q_k(i)$ and
$p_i(k)$ for $0\le i,k\le m$ such that
\begin{align}
\Ei_k = \frac{1}{|\mathcal{X}|}\sum_{i=0}^m q_k(i)\A_i,
\qquad
\A_i = \sum_{k=0}^m p_i(k)\Ei_k.
\end{align}
Moreover, for all $0\le i,j\le m$,
\begin{align}
\label{eq: orthogonality of primitive idempotents}
\Ei_i \Ei_j = \delta_{i,j} \Ei_i, \text{ where } \delta_{i,j} \triangleq \begin{cases}
1 & \text{if } i = j, \\
0 & \text{otherwise.}
\end{cases}
\end{align}
\end{theorem}

Define the matrices $\mathbf{P}\triangleq[p_i(k)]_{i,k=0}^m$ and
$\mathbf{Q}\triangleq[q_k(i)]_{k,i=0}^m$.
The matrices $\mathbf{P}$ and $\mathbf{Q}$ are called the
\emph{first eigenmatrix} and the \emph{second eigenmatrix} of the
association scheme, respectively, and they satisfy
\begin{align}
\mathbf{P}\mathbf{Q}=\mathbf{Q}\mathbf{P}=|\mathcal{X}|\,\I.
\end{align}

\medskip
Association schemes may be viewed as a natural generalization of
distance-regular graphs. In particular, every distance-regular graph
naturally gives rise to an association scheme, called its
\emph{distance scheme}. The Hamming scheme arises in this way from
the Hamming graph $\mathrm{H}(n,q) = \Gilbert{q,n,2}$.
Introductions to association schemes can be found, e.g., in
\cite{Delsarte1973,DelsarteLevenshtein1998} and in
\cite[Chapter~21]{MacWilliamsSloane77}.

\subsection{The Hamming scheme and Krawtchouk polynomials}
\label{subsection: The Hamming scheme and Krawtchouk polynomials}

\begin{definition}[The Hamming scheme]
The Hamming association scheme $\mathcal{H}$ is an $(n+1)$-class symmetric
association scheme on the set $\Field_q^n$, with adjacency matrices
$\{\A_0, \A_1, \ldots, \A_n\}$. For $0 \le i \le n$, the matrix $\A_i$
is the \emph{distance-$i$ matrix}, defined by
\[
(\A_i)_{\mathbf{u},\mathbf{v}} =
\begin{cases}
1 & \text{if } \mathbf{u},\mathbf{v} \in \Field_q^n \text{ differ in exactly } i \text{ coordinates},\\
0 & \text{otherwise}.
\end{cases}
\]
\end{definition}

The dual basis of the Hamming scheme is given in terms of the Krawtchouk polynomials.

\begin{definition}[Krawtchouk polynomials]  \cite[Eq.~(16)]{DelsarteLevenshtein1998}
\label{def:krawtchouk_polynomials}
Let $q,n \in \mathbb{N}$, $q \geq 2$, and $0 \leq i,\ell \leq n$.
The \emph{$q$-ary Krawtchouk polynomial} $K_\ell(i;n,q)$ is defined by
\begin{align}
\label{eq: Krawtchouk polynomial}
K_\ell(i;n,q) = \sum_{r=0}^{\ell} \binom{i}{r} \, \binom{n-i}{\ell-r} \, (-1)^r \, (q-1)^{\,\ell-r}.
\end{align}
Note that $K_\ell(i;n,q)$ is a polynomial of degree $\ell$ in $i$, since for every $r \in \{0,1,\ldots,\ell\}$
the product $\binom{i}{r}\binom{n-i}{\ell-r}$ is itself a polynomial of degree $\ell$ in $i$.
\end{definition}

\begin{theorem}
\label{thm: P and Q matrices of Hamming scheme}
\cite[Example 1]{DelsarteLevenshtein1998}.
The entries of the first and second eigenmatrices $\mathbf{P} \triangleq [p_i(k)]_{i,k=0}^{m}$ and
$\mathbf{Q} \triangleq [q_k(i)]_{k,i=0}^{m}$ of the Hamming scheme are, respectively, given by
\begin{align}
p_i(k) = K_i(k;n,q), \quad q_k(i) = K_k(i;n,q).
\end{align}
In particular, we have
\begin{align}
& \A_i = \sum_{k=0}^{n} K_i(k;n,q) \, \Ei_k, \\
& \Ei_k = \frac{1}{q^n} \sum_{i=0}^{n} K_k(i;n,q) \, \A_i.
\end{align}
\end{theorem}

\begin{theorem}[Properties of Krawtchouk polynomials]
\label{thm: properties of krawtchouk polynomials}
Let $n \in \naturals$, $q \ge 2$, and let $K_l(i;n,q)$ denote the $q$-ary Krawtchouk
polynomials for $0 \le l,i \le n$. Then, the following properties hold:
\begin{enumerate}
\item \label{item: generating function} \cite[Proposition 1.2]{Coleman2011}:
The Krawtchouk polynomials satisfy the generating function identity
\begin{align}
\sum_{k=0}^{n} K_k(i;n,q) \, z^k = \bigl(1+(q-1)z \bigr)^{n-i} (1-z)^i.
\end{align}
\item \label{item: krawtchouk sum} \cite[Theorem 2.1]{Coleman2011}:
For every $i, d \in [n]$, the following summation formula holds
\begin{align}
\sum_{k=0}^{d} K_k(i;n,q) = K_d(i-1;n-1,q).
\end{align}
\end{enumerate}
\end{theorem}

\begin{theorem}[Character-theoretic representation of Krawtchouk polynomials]
\label{thm: character theory definition}
Let $p$ be a prime, let $q = p^m$, and let $0 \le i,k \le n$.
For $\mathbf{x} \in \Field_q^n$ with $\wt{\mathbf{x}} = i$, the following holds:
\begin{enumerate}
\item \cite[Eq.~(29)]{Polyanskiy2019}:
The Krawtchouk polynomial admits the representation
\begin{align}
K_k(i;n,q) =
\sum_{\substack{\mathbf{y} \in \Field_q^n \\ \wt{\mathbf{y}} = k}}
\chi_{\mathbf{y}}(\mathbf{x}),
\end{align}
where $\chi_{\mathbf{y}}$ is the additive character of $\Field_q^n$ indexed by $\mathbf{y}$
and defined by (see \cite[Theorem~5.7]{LidlNiederreiter96})
\[
\chi_{\mathbf{y}}(\mathbf{x}) =
\exp\left(\frac{2\pi i}{p} \cdot \Tr\,\big(\langle \mathbf{x}, \mathbf{y} \rangle \big) \right),
\]
$\langle \mathbf{x}, \mathbf{y} \rangle := \sum_{j=1}^n x_j y_j$ denotes the standard inner product
over $\Field_q$, and $\Tr : \Field_q \to \Field_p$ is the trace function given by
\[
\Tr(a) = \sum_{j=0}^{m-1} a^{p^j},  \qquad a \in \Field_q.
\]
\item The trace map $\Tr : \Field_q \to \Field_p$ is $\Field_p$-linear and surjective
\cite[Theorem~2.23]{LidlNiederreiter96}.
\end{enumerate}
\end{theorem}

\begin{remark}
\label{rem: binary krawtchouk polynomials}
For $q=2$, the Krawtchouk polynomials defined in \eqref{eq: Krawtchouk polynomial}
reduce to the binary Krawtchouk polynomials
\begin{align}
\label{eq: binary}
K_j(i) \triangleq K_j(i;n,2) = \sum_{r=0}^{j} (-1)^r \, \binom{i}{r} \, \binom{n-i}{j-r},
\qquad j \in \{0, 1, \ldots, n\}.
\end{align}
\end{remark}

\subsection{Spectral properties of the complement of the Gilbert graph}
\label{subsection: Spectrum of the complement of the Gilbert graph}

The theory of association schemes enables a characterization of the spectral properties of the complement of the Gilbert graph.
The next result, stated in Theorem~\ref{thm: spectrum of complement of Gilbert graph}, provides such a characterization.
Since this result appears to be new, we include a proof for completeness.

\begin{theorem}
\label{thm: spectrum of complement of Gilbert graph}
Let $\Gr{G} = \overline{\Gilbert{q,n,d}}$ be the complement of the Gilbert graph.
\begin{enumerate}
\item The adjacency matrix of $\Gr{G}$ admits the decomposition
\begin{align}
\label{eq: decomposition}
\Adjacency(\Gr{G}) = \sum_{j=0}^{n} \gamma_j \Ei_j,
\end{align}
where the scalars $\gamma_j$ are the eigenvalues of $\Adjacency(\Gr{G})$, given by
\begin{align}
\label{eq: 20.03.26}
\gamma_0 = \Delta(\Gr{G}), \quad \gamma_j = - K_{d-1}(j-1;n-1,q), \, \forall 1 \le j \le n,
\end{align}
and $\Delta(\Gr{G}) = \sum_{i=d}^{n} \binom{n}{i} (q-1)^i$ denotes the degree of the regular graph $\Gr{G}$.
\item \label{item: distinct eigenvalues} $\gamma_1 < \gamma_m$ for every $m$ with $2 \leq m \leq n-1$.
\item \label{item: equal eigenvalues} $\gamma_1 = \gamma_n$ if and only if $q=2$ and $d$ is odd.
\item \label{item: polynomial expression} For every polynomial $p$,
\begin{align}
\label{eq:12.02.26}
p\bigl(\Adjacency(\Gr{G})\bigr) = \sum_{j=0}^{n} p(\gamma_j) \, \Ei_j.
\end{align}
\end{enumerate}
\end{theorem}

\begin{proof}
Let $\Gr{G} = \overline{\Gilbert{q,n,d}}$ be the complement of the Gilbert graph, $1<d<n$.
\begin{enumerate}
\item The adjacency matrix of $\Gr{G}$ can be expressed as follows:
\begin{align}
\Adjacency(\Gr{G}) &= \sum_{i=d}^{n} \A_i \nonumber \\
&= \sum_{i=d}^{n} \left(\sum_{j=0}^{n} K_i(j; n,q) \Ei_j\right) \nonumber \\
&= \sum_{j=0}^{n} \left( \sum_{i=d}^{n} K_i(j;n,q) \right) \Ei_j \nonumber \\
&\triangleq \sum_{j=0}^{n} \gamma_j \Ei_j,
\end{align}
where $\{\A_i\}_{i=0}^{n}$ are the distance $i$ matrices of the Hamming scheme, and $\{\Ei_j\}_{j=0}^{n}$
is the dual basis of its Bose-Mesner algebra. The values $\{\gamma_j\}_{j=0}^{n}$ are given by
\begin{align}
\label{eq: gammaj definition}
\gamma_j \triangleq \sum_{i=d}^{n} K_i(j; n,q) = \sum_{i=0}^{n} K_i(j; n,q) - \sum_{i=0}^{d-1} K_i(j; n,q),
\end{align}
and form the spectrum of $\Gr{G}$ with respect to its adjacency matrix.
The largest eigenvalue $\gamma_0$ is the degree of regularity of $\Gr{G}$, i.e.
\begin{align}
\label{eq: gamma0 expression}
\gamma_0 = \Delta(\Gr{G}) = \sum_{i=d}^{n} \binom{n}{i} (q-1)^i.
\end{align}
By evaluating the generating function at $z=1$ from Theorem~\ref{thm: properties of krawtchouk polynomials}, we can conclude that for every $0 < j \le n$,
\begin{align}
\sum_{i=0}^{n} K_i(j; n,q) = 0.
\end{align}
By the Summation formula in Theorem~\ref{thm: properties of krawtchouk polynomials}, for every $1 \le j \le n$,
\begin{align}
\label{eq: gamma j expression}
\gamma_j = - \sum_{i=0}^{d-1} K_i(j; n,q) = - K_{d-1}(j-1;n-1,q).
\end{align}
Together, equalities \eqref{eq: gamma0 expression} and \eqref{eq: gamma j expression} conclude the proof.

\item We compare $\gamma_1$ and $\gamma_m$ for $2 \le m \le n-1$. By \eqref{eq: gamma j expression},
\begin{align}
\gamma_1 - \gamma_m
&= K_{d-1}(m-1;n-1,q) - K_{d-1}(0;n-1,q) \\
&= K_{d-1}(m-1;n-1,q) - (q-1)^{d-1} \binom{n-1}{d-1}.
\end{align}
Let $\mathbf{x} \in \Field_q^{n-1}$ be a vector with Hamming weight $\wt{\mathbf{x}} = m-1$. By the character-theoretic representation of Krawtchouk polynomials (Theorem~\ref{thm: character theory definition}),
\begin{align}
K_{d-1}(m-1;n-1,q)
= \sum_{\mathbf{y} \in \Field_q^{n-1} \,:\, \wt{\mathbf{y}} = d-1}
\omega^{\trace{\inner{\mathbf{x}}{\mathbf{y}}}},
\end{align}
where $\omega = e^{2\pi i/p}$. Since $|\omega^{\trace{\inner{\mathbf{x}}{\mathbf{y}}}}|=1$, we have
\begin{align}
\label{eq: krawtchouk upper bound}
K_{d-1}(m-1;n-1,q)
\le \sum_{\mathbf{y} \in \Field_q^{n-1} \,:\, \wt{\mathbf{y}} = d-1} 1
= (q-1)^{d-1} \binom{n-1}{d-1}.
\end{align}
We claim that the inequality is strict. Indeed, equality would require that
\[
\omega^{\trace{\inner{\mathbf{x}}{\mathbf{y}}}} = 1
\quad \text{for all } \mathbf{y} \text{ with } \wt{\mathbf{y}} = d-1,
\]
i.e., the additive character $\mathbf{y} \mapsto \omega^{\trace{\inner{\mathbf{x}}{\mathbf{y}}}}$ is identically equal to $1$ on this set. However, when $\mathbf{x} \neq 0$, this character is nontrivial (see, e.g., \cite[Ch.~5]{LidlNiederreiter96} or \cite[Ch.~5]{MacWilliamsSloane77}), and therefore it cannot be identically equal to $1$ on all vectors of a given weight. Since $\wt{\mathbf{x}} = m-1 \ge 1$, we conclude that the inequality in \eqref{eq: krawtchouk upper bound} is strict. Hence,
\begin{align}
\gamma_1 - \gamma_m < 0, \qquad \forall\, 2 \le m \le n-1.
\end{align}

\item In the cases where $j=1$ or $j=n$, the eigenvalues $\gamma_1$ and $\gamma_n$ can be expressed explicitly as follows:
\begin{align}
\gamma_1 &= - K_{d-1}(0;n-1,q) = -(q-1)^{d-1} \binom{n-1}{d-1} \\
\gamma_n &= - K_{d-1}(n-1;n-1,q) = (-1)^d \binom{n-1}{d-1}.
\end{align}
Hence,
\begin{align}
\label{eq: gamma 1 equals gamma n}
\gamma_1 = \gamma_n \iff q=2 \text{ and } d \text{ is odd.}
\end{align}

\item \label{subtheorem: part 4}
By Equations~\eqref{eq: orthogonality of primitive idempotents} and \eqref{eq: decomposition}, every polynomial
$p(\Adjacency(\Gr{G})) = \sum_{k=0}^{r} \alpha_k \Adjacency(\Gr{G})^k$ of the adjacency matrix can be expressed
in terms of the orthogonal projections as follows:
\begin{align}
\label{eq: polynomial of adjacency matrix}
p(\Adjacency(\Gr{G})) &= \sum_{k=0}^{r} \alpha_k \Adjacency(\Gr{G})^k
= \sum_{k=0}^{r} \alpha_k \left( \sum_{j=0}^{n} \gamma_j \Ei_j \right)^k
= \sum_{k=0}^{r} \sum_{j=0}^{n} \alpha_k \gamma_j^k \Ei_j = \sum_{j=0}^{n} p(\gamma_j) \Ei_j.
\end{align}
\end{enumerate}
\end{proof}

\begin{remark}
The eigenvalues of the Gilbert graph $\Gr{H} \triangleq \Gilbert{q,n,d}$ are given in \cite{YeZLWYMa2023}:
\begin{align}
\label{eq: spectrum of gilbert}
\bigl(\Delta(\Gr{H}), K_{d-1}(0;n-1,q)-1, \ldots, K_{d-1}(n-1;n-1,q)-1\bigr).
\end{align}
Since $\Gilbert{q,n,d}$ is regular with degree $\Delta(\Gr{H})$, by \cite[Section 1.3.2]{BrouwerHaemers2012}
the eigenvalues of $\overline{\Gilbert{q,n,d}}$ are given as a function of $|\V{\Gr{H}}| = q^n$ and the eigenvalues of $\Gr{H}$ as follows:
\begin{align}
\bigl(q^n - 1 - \Delta(\Gr{H}), -K_{d-1}(0;n-1,q), \ldots, -K_{d-1}(n-1;n-1,q) \bigr).
\end{align}
This expression is consistent with the spectrum of $\overline{\Gilbert{q,n,d}}$ given in Theorem~\ref{thm: spectrum of complement of Gilbert graph}.
Parts~\ref{item: distinct eigenvalues}, \ref{item: equal eigenvalues}, and \ref{item: polynomial expression} of
Theorem~\ref{thm: spectrum of complement of Gilbert graph} will be used in the proof of Theorem~\ref{thm: main result complement}, specifically in
the proof of Lemmata~\ref{lem: complement of gilbert not edge transitive part 1} and~\ref{lem: complement of gilbert not edge transitive part 2}.
\end{remark}

\section{Proof of Theorem~\ref{thm: main result complement}}
\label{sec: transitive complement gilbert graphs}

Let $\Gr{G} = \overline{\Gilbert{q,n,d}}$ be the complement of the Gilbert graph. Since $\Gr{G}$ is a Cayley graph, it is vertex-transitive,
so it is edge-transitive if and only if for every two neighbors of the zero vector, $\mathbf{u},\mathbf{v} \in \Neighbors(\mathbf{0})$, there
exists an automorphism $\varphi \in \Aut(\Gr{G})$ that maps the edge $\{\mathbf{0},\mathbf{u}\}$ to $\{\mathbf{0},\mathbf{v}\}$.
The proof of Theorem~\ref{thm: main result complement} is divided into several lemmas. We first show that $\Gr{G}$ is not edge-transitive under
certain parameter choices. Afterwards, we prove that in the remaining cases $\Gr{G}$ is indeed edge-transitive.

\begin{lemma}
\label{lem: complement of gilbert not edge transitive part 1}
Let $\Gr{G} = \overline{\Gilbert{q,n,d}}$. If $q\ne 2$ and $1 < d < n$, or if $q=2$, $1 < d < n$ and $d$ is even, then $\Gr{G}$ is not edge-transitive.
\end{lemma}

\begin{proof}
By the connectivity of the graph $\Gr{G}$, the largest eigenvalue $\gamma_0$ of its adjacency matrix is simple (see, e.g., \cite[Proposition~12.1.1]{CioabaM22});
in particular, $\gamma_1 < \gamma_0$.
By Part~2 of Theorem~\ref{thm: spectrum of complement of Gilbert graph}, it also holds that $\gamma_1 < \gamma_m$ for all $m \in \{2, \ldots, n-1\}$.
Furthermore, under the hypotheses of the lemma that $q \neq 2$ or $d$ is even if $q=2$, Part~3 of Theorem~\ref{thm: spectrum of complement of Gilbert graph}
ensures that $\gamma_1 \neq \gamma_n$. Hence, the polynomial $p \colon \Reals \to \Reals$ given by
\begin{align}
p(t) \triangleq \prod_{m \neq 1} \frac{t - \gamma_m}{\gamma_1 - \gamma_m}, \qquad t \in \Reals,
\end{align}
is well defined under the hypotheses of the lemma. By construction, it satisfies $p(\gamma_1) = 1$ and $p(\gamma_m) = 0$ for every $m \neq 1$ (i.e.,
$m \in \{0, 2, \ldots, n\}$).
Consequently, by Theorem~\ref{thm: spectrum of complement of Gilbert graph} (see \eqref{eq:12.02.26}), we obtain
\begin{align}
p(\Adjacency(\Gr{G})) = \Ei_1,
\end{align}
whose explicit expression is given by
\begin{align}
\Ei_1 = \frac{1}{q^n} \sum_{i=0}^{n} K_1(i;n,q) \A_i.
\end{align}
Therefore, for every two vertices $\mathbf{u},\mathbf{v} \in \Field_q^n$ with Hamming distance $\wt{\mathbf{u},\mathbf{v}} = i$, the entry
that corresponds to $(\mathbf{u},\mathbf{v})$ in the matrix $\Ei_1$ is given by
\begin{align}
\label{eq: E1 uv entry}
(\Ei_1)_{\mathbf{u},\mathbf{v}} = \frac{K_1(i;n,q)}{q^n} = \frac{(q-1)(n - i) - i}{q^n}.
\end{align}

Suppose, to the contrary, that $\Gr{G}$ is edge-transitive. Let $\mathbf{u},\mathbf{v} \in \Neighbors(\mathbf{0})$ be two vertices such that
$\wt{\mathbf{u}} = n$ and $\wt{\mathbf{v}} = n-1$. (Indeed, since $\Gr{G}$ is the complement of $\Gilbert{q,n,d}$ with $d < n$, it follows
from the assumptions of the lemma that both $\mathbf{u}$ and $\mathbf{v}$ are adjacent to the vertex corresponding to the all-zero vector.)
By the edge transitivity assumption, there exists an automorphism $\varphi \in \Aut(\Gr{G})$ that maps the edge $\{\mathbf{0},\mathbf{u}\}$
to $\{\mathbf{0},\mathbf{v}\}$. Let $P_{\varphi}$ be the permutation matrix that corresponds to an automorphism $\varphi$ in $\Gr{G}$. Then,
\begin{align}
\label{eq: E1 = P E1 P^-1}
\Ei_1 = p(\Adjacency(\Gr{G})) = P_{\varphi} p(\Adjacency(\Gr{G})) P_{\varphi}^{-1} = P_\varphi \Ei_1 P_\varphi^{-1},
\end{align}
which implies that the following relation is satisfied:
\begin{align}
\label{eq: E1 uv - E1 vv = 0}
(\Ei_1)_{\mathbf{0},\mathbf{v}} - (\Ei_1)_{\mathbf{0},\mathbf{u}} = 0.
\end{align}
Setting $i=n$ and $i=n-1$ in \eqref{eq: E1 uv entry} (as the Hamming weights of $\mathbf{u}$ and $\mathbf{v}$, respectively) gives
\begin{align}
&(\Ei_1)_{\mathbf{0},\mathbf{u}} = -\frac{n}{q^n} \\[0.1cm]
&(\Ei_1)_{\mathbf{0},\mathbf{v}} = \frac{q - n}{q^n} \\
\implies &(\Ei_1)_{\mathbf{0},\mathbf{v}} - (\Ei_1)_{\mathbf{0},\mathbf{u}} = \frac{1}{q^{n-1}} \ne 0,
\end{align}
contradicting equation \eqref{eq: E1 uv - E1 vv = 0}. Hence, $\Gr{G}$ cannot be edge-transitive under the hypotheses of
Lemma~\ref{lem: complement of gilbert not edge transitive part 1}.
\end{proof}

\begin{lemma}
\label{lem: complement of gilbert not edge transitive part 2}
Let $\Gr{G} = \overline{\Gilbert{q,n,d}}$. If $q=2$ and $d < n-1$ is an odd number, then $\Gr{G}$ is not edge-transitive.
\end{lemma}

\begin{proof}
The proof is similar to that of Lemma~\ref{lem: complement of gilbert not edge transitive part 1}.
By Theorem~\ref{thm: spectrum of complement of Gilbert graph}, the polynomial
\begin{align}
p(t) \triangleq \prod_{m \notin \{1,n\}} \frac{t - \gamma_m}{\gamma_1 - \gamma_m}
\end{align}
is well defined. Evaluating the polynomial at $\gamma_i$ for $0 \le i \le n$ gives
$p(\gamma_1) = p(\gamma_n) = 1$ and $p(\gamma_m) = 0$ for every $m \notin \{1,n\}$.
Thus, by equality~\eqref{eq:12.02.26},
\begin{align}
p(\Adjacency(\Gr{G})) = \Ei_1 + \Ei_n.
\end{align}
An explicit expression for $\Ei_n$ is given by
\begin{align}
\Ei_n &= \frac{1}{2^n} \sum_{i=0}^{n} K_n(i) \A_i \nonumber \\
\label{eq: En expression}
&= \frac{1}{2^n} \sum_{i=0}^{n} (-1)^i \A_i.
\end{align}
where equality \eqref{eq: En expression} holds since, by \eqref{eq: binary},
\[
K_n(i) := K_n(i;n,2) = (-1)^i, \qquad i \in \{0, 1, \ldots, n\}.
\]
Let $\mathbf{u}, \mathbf{v}\in \Neighbors(\Zero)$ be two vertices such that $\wt{\mathbf{u}} = n$ and $\wt{\mathbf{v}} = n-2$
(note that, since $d<n-1$, both $\mathbf{u}$ and $\mathbf{v}$ are neighbors of the vertex corresponding to the all-zero vector
in $\Gr{G} = \overline{\Gilbert{q,n,d}}$).
Suppose to the contrary that $\Gr{G}$ is edge-transitive. Then, there exists an automorphism $\varphi \in \Aut(\Gr{G})$ that maps
the edge $\{\Zero,\mathbf{u}\}$ to the edge $\{\Zero,\mathbf{v}\}$. Let $P_{\varphi}$ be the permutation matrix corresponding to
the automorphism $\varphi$. Since $\varphi$ is a graph automorphism, we get
\begin{align}
\label{eq: E1+En = P (E1+En)}
\Ei_1 + \Ei_n = p(\Adjacency(\Gr{G})) = P_{\varphi} p(\Adjacency(\Gr{G})) P_{\varphi}^{-1} = P_\varphi (\Ei_1 + \Ei_n) P_\varphi^{-1}.
\end{align}
Hence, the following relation must hold:
\begin{align}
\label{eq: E1+En uv - E1+En vv = 0}
(\Ei_1 + \Ei_n)_{\Zero,\mathbf{v}} - (\Ei_1 + \Ei_n)_{\Zero,\mathbf{u}} = 0.
\end{align}
However, by equations \eqref{eq: E1 uv entry}, \eqref{eq: En expression}, and \eqref{eq: E1+En = P (E1+En)}, we obtain
\begin{align}
&(\Ei_1 + \Ei_n)_{\Zero,\mathbf{u}} = \frac{(n - 2n) + (-1)^{n}}{2^n} = \frac{-n + (-1)^{n}}{2^n}, \\
&(\Ei_1 + \Ei_n)_{\Zero,\mathbf{v}} = \frac{(n - 2(n-2)) + (-1)^{n-2}}{2^n} = \frac{-n + 4 + (-1)^{n}}{2^n}, \\[0.1cm]
\implies &(\Ei_1 + \Ei_n)_{\Zero,\mathbf{v}} - (\Ei_1 + \Ei_n)_{\Zero,\mathbf{u}} = 2^{-(n-2)},
\end{align}
which contradicts equation \eqref{eq: E1+En uv - E1+En vv = 0}. Hence, the graph $\Gr{G}$ cannot be edge-transitive.
\end{proof}

\begin{lemma}
\label{lem: gilbert complement distance transitivity n}
Let $\Gr{G}=\overline{\Gilbert{q,n,n}}$ be the complement of the Gilbert graph with $d=n$. Then,
\begin{enumerate}
\item $\Gr{G}$ is edge-transitive.
\item It is distance transitive if and only if $q=2$ or $n=2$.
\end{enumerate}
\end{lemma}

\begin{proof}
We first start by proving that  $\Gr{G}$  is edge-transitive, followed by the proof that $\Gr{G}$ is
distance-transitive if $q=2$ or $n=2$. Then we show that if the latter condition is violated, then
$\Gr{G}$ is not distance-transitive.
\begin{enumerate}
\item If $q=2$, the only neighbor of $\mathbf{0}$ is $\mathbf{1}_n$.
Hence, $\Gr{G}$ is trivially edge-transitive. 

\noindent 
If $q \geq 3$, let $\mathbf{u}, \mathbf{v} \in \Field_q^n$ have all entries nonzero.
Define $\varphi \colon \Field_q^n \to \Field_q^n$ by
\[
\varphi(x_1, \dots, x_n) = (u_1^{-1} x_1 v_1, \dots, u_n^{-1} x_n v_n).
\]
Then $\varphi$ is a bijection with inverse
\[
\varphi^{-1}(y_1, \dots, y_n)
= (u_1 y_1 v_1^{-1}, \dots, u_n y_n v_n^{-1}).
\]
Each coordinate is multiplied by a nonzero field element.
Hence, Hamming distances are preserved. Therefore, $\varphi \in \Aut(\Gr{G})$
(i.e., the mapping $\varphi$ is a graph automorphism of $\Gr{G}$).
Moreover, $\varphi(\mathbf{0}) = \mathbf{0}$ and $\varphi(\mathbf{u}) = \mathbf{v}$.
Thus, $\{\mathbf{0}, \mathbf{u}\}$ maps to $\{\mathbf{0}, \mathbf{v}\}$.
Hence, $\Gr{G}$ is edge-transitive.

\item If $q=2$, then $\Gr{G}$ is a disjoint union of $K_2$'s.
Thus, it is distance-transitive.

\smallskip
\noindent 
Assume $q \ge 3$, and consider the cases where $n=2$ or $n \geq 3$ separately.

\noindent
\textbf{Case $n=2$.}
The vertices at distance $2$ from $\mathbf{0}$ are those of weight $1$.
For any two such vertices $u,v \in \mathbb{F}_q^n$, there exists a coordinate
permutation $\pi$ such that $u$ and $\pi(v)$ have the same support.
Composing with $\varphi$ gives distance-transitivity.

\noindent
\textbf{Case $n>2$.}
We show that $\Gr{G}$ is not distance-transitive. Let
\[
\mathbf{u} = (1,0,0,\dots,0), \qquad
\mathbf{v} = (0,1,1,\dots,1).
\]
Both are at distance~$2$ from $\mathbf{0}$.
Assume, for contradiction, that $\Gr{G}$ is distance-transitive.
Then, there exists $\varphi \in \Aut(\Gr{G})$ fixing $\mathbf{0}$ and mapping $\mathbf{u}$ to $\mathbf{v}$.
Define
\[
C(\mathbf{x},\mathbf{y}) = |\Neighbors(\mathbf{x}) \cap \Neighbors(\mathbf{y})|.
\]
In a distance-transitive graph, $C(\mathbf{x},\mathbf{y})$ depends only on $d(\mathbf{x},\mathbf{y})$.
Hence,
\[
C(\mathbf{0},\mathbf{u}) = C(\mathbf{0},\mathbf{v}).
\]

\noindent
\textbf{Computation of $C(\mathbf{0},\mathbf{u})$.}
A common neighbor $z$ satisfies $z_i \neq 0$ for all $i \in [n]$, and $z_1 \neq 1$. Thus,
\[
z_1 \in \Field_q \setminus \{0,1\}, \quad z_i \in \Field_q \setminus \{0\} \ (i \ge 2).
\]
Hence,
\[
C(\mathbf{0},\mathbf{u}) = (q-2)(q-1)^{n-1}.
\]

\noindent
\textbf{Computation of $C(\mathbf{0},\mathbf{v})$.}
Now $z$ satisfies $z_i \neq 0$ for all $i$, and $z_i \neq 1$ for $i \ge 2$. Thus,
\[
z_1 \in \Field_q \setminus \{0\}, \quad
z_i \in \Field_q \setminus \{0,1\} \ (i \ge 2).
\]
Hence,
\[
C(\mathbf{0},\mathbf{v}) = (q-1)(q-2)^{n-1}.
\]
For $n>2$, $C(\mathbf{0},\mathbf{u}) \neq C(\mathbf{0},\mathbf{v})$, which leads to a contradiction.
Therefore, $\Gr{G}$ is not distance-transitive.
\end{enumerate}
\end{proof}

\begin{lemma}
\label{lem: gilbert complement distance transitivity n-1}
If $(q,d) = (2,n-1)$ and $n$ is even, then $\Gr{G} = \overline{\Gilbert{2,n,n-1}}$ is distance-transitive,
and hence edge-transitive.
\end{lemma}

\begin{proof}
The distance layers from $\mathbf{0}$ are determined by Hamming weights.
In particular, the diameter of $\Gr{G}$ is $\lceil n/2 \rceil$, and the
distance-$k$ layer consists of all vectors $\mathbf{v}$ with
\[
\wt{\mathbf{v}} \in
\begin{cases}
\{n-k+1,\, n-k\}, & \text{if $k$ is odd},\\
\{k-1,\, k\}, & \text{if $k$ is even}.
\end{cases}
\]
Since $n$ is even, any two vectors in the same layer have weights that are either equal or differ by $1$.
Let $1 \le k \le \frac{n}{2}$, and let $\mathbf{u}, \mathbf{v} \in \V{\Gr{G}}$ satisfy
\[
\DG(\mathbf{0}, \mathbf{u}) = \DG(\mathbf{0}, \mathbf{v}) = k.
\]
If $\wt{\mathbf{u}} = \wt{\mathbf{v}}$, then there exists a coordinate permutation $\pi$ that maps
$\mathbf{u}$ to $\mathbf{v}$. Since $\pi$ fixes $\mathbf{0}$, it is an automorphism that maps
$\{\mathbf{u}, \mathbf{0} \}$ to $\{\mathbf{v}, \mathbf{0}\}$.

Assume now that $\wt{\mathbf{u}} \ne \wt{\mathbf{v}}$. Since $n$ is even, we may assume
\[
\wt{\mathbf{u}} = 2i, \qquad \wt{\mathbf{v}} = 2i-1,
\]
for some $1 \le i \le \frac{n}{2}$. Since coordinate permutations are automorphisms of $\Gr{G}$,
it suffices to construct an automorphism mapping
$(\mathbf{0}, \tilde{\mathbf{u}})$ to $(\mathbf{0}, \tilde{\mathbf{v}})$, where
\begin{align}
\tilde{\mathbf{u}} &= (\AllOne_{2i-1}, \Zero_{n-2i}, 1), \\
\tilde{\mathbf{v}} &= (\AllOne_{2i-1}, \Zero_{n-2i}, 0).
\end{align}
Define $\varphi \colon \Field_2^n \to \Field_2^n$ by
\begin{align}
\varphi(x_1, \ldots, x_n) &= (x_1, \ldots, x_{n-1}, \parity(\mathbf{x})), \\
\parity(\mathbf{x}) &= \wt{\mathbf{x}} \bmod 2.
\end{align}
Then $\varphi$ is an involution, and hence bijective. Moreover, $\varphi(\mathbf{0}) = \mathbf{0}$.
We show that $\varphi$ is an automorphism of $\Gr{G}$. Let $\mathbf{x}, \mathbf{y} \in \Field_2^n$.
\begin{itemize}
\item If $\Dh(\mathbf{x}, \mathbf{y}) = n$, then the first $n-1$ coordinates of $\varphi(\mathbf{x})$
and $\varphi(\mathbf{y})$ differ. Hence,
\[
\Dh(\varphi(\mathbf{x}), \varphi(\mathbf{y})) \ge n-1,
\]
so $\{\varphi(\mathbf{x}), \varphi(\mathbf{y})\} \in \E{\Gr{G}}$.

\item If $\Dh(\mathbf{x}, \mathbf{y}) = n-1$, then $x_k = y_k$ for a unique $k$. Since $n$ is even,
\[
\parity(\mathbf{x}) + \parity(\mathbf{y}) = \parity(\mathbf{x} + \mathbf{y}) = n-1 \equiv 1 \pmod{2},
\]
so $\parity(\mathbf{x}) \ne \parity(\mathbf{y})$. Hence,
\[
(\varphi(\mathbf{x}))_i \ne (\varphi(\mathbf{y}))_i \quad \text{for all } i \ne k,
\]
and therefore
\[
\Dh(\varphi(\mathbf{x}), \varphi(\mathbf{y})) \ge n-1.
\]
Thus, $\{\varphi(\mathbf{x}), \varphi(\mathbf{y})\} \in \E{\Gr{G}}$.
\end{itemize}
Since $\varphi$ is bijective, it follows that $\varphi \in \Aut(\Gr{G})$.
Finally, since $\wt{\tilde{\mathbf{u}}} = 2i$ and $\wt{\tilde{\mathbf{v}}} = 2i-1$, we have
$\parity(\tilde{\mathbf{u}}) = 0$ and $\parity(\tilde{\mathbf{v}}) = 1$. Hence,
$\varphi(\tilde{\mathbf{u}}) = \tilde{\mathbf{v}}$. This completes the proof of the lemma.
\end{proof}

\subsection*{Completion of the proof of Theorem~\ref{thm: main result complement}}
Summarizing Lemmas~\ref{lem: complement of gilbert not edge transitive part 1}--\ref{lem: gilbert complement distance transitivity n-1},
we obtain the following:
\begin{enumerate}
\item If $d=n$, then $\overline{\Gilbert{q,n,d}}$ is edge-transitive.
In this case, it is distance-transitive if and only if $q=2$ or $n=2$
(Lemma~\ref{lem: gilbert complement distance transitivity n}).

\item If $d=n-1$ and $n$ is even, then $\overline{\Gilbert{q,n,d}}$ is distance-transitive,
and hence edge-transitive (Lemma~\ref{lem: gilbert complement distance transitivity n-1}).

\item In all remaining cases, the graph is not edge-transitive.
Specifically, this holds if $q \ne 2$ and $1 < d < n$, or if $q=2$ and $1 < d < n-1$
(Lemma~\ref{lem: complement of gilbert not edge transitive part 1}),
or if $d=n-1$ and $n$ is odd
(Lemma~\ref{lem: complement of gilbert not edge transitive part 2}).
In these cases, the graph is also not distance-transitive, since distance-transitivity
implies edge-transitivity.
\end{enumerate}
This completes the proof of Theorem~\ref{thm: main result complement}.

\section{\texorpdfstring{The Lov\'{a}sz $\vartheta$ function}{The Lovasz theta function}}
\label{sec: lovasz theta function}

The Lov\'{a}sz $\vartheta$-function of a graph, denoted $\vartheta(\Gr{G})$, is a graph invariant
that exhibits deep connections to several fundamental graph invariants, including the independence
number $\indnum{\Gr{G}}$, clique number $\clnum{\Gr{G}}$, chromatic number $\chrnum{\Gr{G}}$,
fractional chromatic number $\fchrnum{\Gr{G}}$, Shannon capacity $\Theta(\Gr{G})$, and others
(see, e.g., \cite{BallaJS2024, Orel23}). For a comprehensive overview, we refer to \cite{Sason2024},
and for a unified framework encompassing these invariants, see \cite{Fritz21}. Owing to its central
role, the Lov\'{a}sz $\vartheta$-function has also been generalized in \cite{Roberson16} for the
study of graph homomorphisms. The present section provides necessary background on the Lov\'{a}sz
$\vartheta$-function for Section~\ref{sec: applications}.

\begin{definition}[Lov\'{a}sz $\vartheta$-function]
\label{definition: Lovasz theta function}
Let $\Gr{G}$ be a finite, simple, and undirected graph. Then, the {\em Lov\'{a}sz
$\vartheta$-function of $\Gr{G}$} is defined as
\begin{eqnarray}
\label{eq: Lovasz theta function}
\vartheta(\Gr{G}) \triangleq \min_{\bf{u}, \bf{c}} \, \max_{i \in \V{\Gr{G}}} \,
\frac1{\bigl( {\bf{c}}^{\mathrm{T}} {\bf{u}}_i \bigr)^2} \, ,
\end{eqnarray}
where the minimum is taken over all orthonormal representations $\{{\bf{u}}_i: i \in \V{\Gr{G}} \}$ of $\Gr{G}$
and all unit vectors ${\bf{c}}$.
The minimization on the right-hand side of \eqref{eq: Lovasz theta function} is taken over all ambient
dimensions; however, it suffices to restrict the vectors to $\mathbb{R}^N$, where $N = \card{\V{\Gr{G}}}$.
\end{definition}

The Lov\'{a}sz $\vartheta$-function can be expressed as a solution of an SDP problem.
Let $N \triangleq \card{\V{\Gr{G}}}$ and $\Adjacency = (A_{i,j})_{i,j \in \OneTo{N}}$.
Then, $\vartheta(\Gr{G})$ can be expressed as the solution of the following convex optimization problem:
\vspace*{0.3cm}
\begin{eqnarray}
\label{eq: SDP problem - Lovasz theta-function}
\mbox{\fbox{$
\begin{array}{l}
\text{maximize} \; \; \mathrm{Tr}({\bf{B}} \, {\bf{J}}_N)  \\
\text{subject to} \\
\begin{cases}
{\bf{B}} \succeq 0, \\
\mathrm{Tr}({\bf{B}}) = 1, \\
A_{i,j} = 1  \; \Rightarrow \;  B_{i,j} = 0, \quad i,j \in \OneTo{N}.
\end{cases}
\end{array}$}}
\end{eqnarray}

\medskip
The semidefinite programming (SDP) formulation in \eqref{eq: SDP problem - Lovasz theta-function}
implies that $\vartheta(\Gr{G})$ can be computed to any prescribed accuracy $r$ in time polynomial
in $N$ and $\log \frac{1}{r}$.

\begin{theorem}[Properties of the Lov\'{a}sz $\vartheta$-function]
\label{thm: properties of the Lovasz function}
Let $\Gr{G}$ be a finite, simple, and undirected graph on $N$ vertices. The following properties hold:
\begin{enumerate}
\item \label{item: Lovasz product}
\begin{align}
\vartheta(\Gr{G}) \; \vartheta(\CGr{G}) \geq N,
\end{align}
with an equality if $\Gr{G}$ is vertex-transitive or strongly regular.
\item \label{item: general Lovasz bounds}
The following inequalities hold:
\begin{align}
\label{eq1: 20.03.26}
\indnum{\Gr{G}} &\le \vartheta(\Gr{G}) \le \fchrnum{\CGr{G}} \le \chrnum{\CGr{G}} \\
\label{eq2: 20.03.26}
\clnum{\Gr{G}} &\le \vartheta(\CGr{G}) \le \fchrnum{\Gr{G}} \le \chrnum{\Gr{G}}.
\end{align}
\end{enumerate}
\end{theorem}

\begin{definition}[Maximum cut and surplus]
Let $\Gr{G} = (\V{\Gr{G}}, \E{\Gr{G}})$ be a graph. A \emph{maximum cut} of $\Gr{G}$ is a partition
of the vertex set $\V{\Gr{G}}$ into two disjoint subsets $\set{S}$ and $\set{T}$ that maximizes
the number of edges with one endpoint in $\set{S}$ and the other in $\set{T}$.
\begin{enumerate}
\item
The \emph{max-cut} of $\Gr{G}$, denoted by $\mathrm{mc}(\Gr{G})$, is the maximum number of such edges,
taken over all partitions of $\V{\Gr{G}}$. Equivalently, $\mathrm{mc}(\Gr{G})$ is the maximum number
of edges in a bipartite spanning subgraph of $\Gr{G}$.
\item
The \emph{surplus} of $\Gr{G}$ is defined as
\[
\mathrm{sp}(\Gr{G}) \triangleq \mathrm{mc}(\Gr{G}) - \frac{1}{2}\,\card{\E{\Gr{G}}}.
\]
\end{enumerate}
\end{definition}

\begin{theorem}[Lower bounds on the surplus, \cite{BallaJS2024}]
\label{theorem: surplus LB}
For every graph $\Gr{G}$, its max-cut satisfies
\begin{align}
\label{eq: surplus LB}
\mathrm{sp}(\Gr{G}) \ge \frac{1}{\pi} \cdot \frac{\card{\E{\Gr{G}}}}{\vartheta(\CGr{G}) - 1} \ge 0.
\end{align}
\end{theorem}

\noindent
In particular, any upper bound on $\vartheta(\CGr{G})$ yields a corresponding (possibly weaker)
lower bound on the surplus $\mathrm{sp}(\Gr{G})$.

\medskip
The next theorem bounds the Lov\'{a}sz $\vartheta$-function in terms of the second largest eigenvalue
of the adjacency matrix of the graph.

\begin{theorem}[Bounds on the Lov\'{a}sz $\vartheta$-function of regular graphs, \cite{Sason23}]
\label{thm:bounds on the Lovasz function for regular graphs}
Let $\Gr{G}$ be a $\Delta$-regular graph of order $N$, which is neither complete nor empty.
Then, the following bounds hold for the Lov\'{a}sz $\vartheta$-function of $\Gr{G}$ and its
complement $\CGr{G}$:
\begin{enumerate}[(1)]
\item
\begin{eqnarray}
\label{eq:bounds on the Lovasz function for regular graphs 1}
\frac{N-\Delta+\Eigval{2}{\Gr{G}}}{1+\Eigval{2}{\Gr{G}}} \leq \vartheta(\Gr{G})
\leq -\frac{N \Eigval{N}{\Gr{G}}}{\Delta - \Eigval{N}{\Gr{G}}}.
\end{eqnarray}
\begin{itemize}
\item Equality holds in the leftmost inequality of \eqref{eq:bounds on the Lovasz function for regular graphs 1} if $\CGr{G}$
is both vertex-transitive and edge-transitive, or if $\Gr{G}$ is a strongly regular graph;
\item Equality holds in the rightmost inequality of \eqref{eq:bounds on the Lovasz function for regular graphs 1} if $\Gr{G}$
is edge-transitive, or if $\Gr{G}$ is a strongly regular graph.
\end{itemize}
\item
\begin{eqnarray}
\label{eq:bounds on the Lovasz function for regular graphs 2}
1 - \frac{\Delta}{\Eigval{N}{\Gr{G}}} \leq \vartheta(\CGr{G})
\leq \frac{N \bigl(1+\Eigval{2}{\Gr{G}}\bigr)}{N-\Delta+\Eigval{2}{\Gr{G}}}.
\end{eqnarray}
\begin{itemize}
\item Equality holds in the leftmost inequality of \eqref{eq:bounds on the Lovasz function for regular graphs 2}
if $\Gr{G}$ is both vertex-transitive and edge-transitive, or if $\Gr{G}$ is
a strongly regular graph;
\item Equality holds in the rightmost inequality of \eqref{eq:bounds on the Lovasz function for regular graphs 2}
if $\CGr{G}$ is edge-transitive, or if $\Gr{G}$ is a strongly regular graph.
\end{itemize}
\end{enumerate}
\end{theorem}

Since the Lov\'{a}sz $\vartheta$-function is computable in polynomial time via SDP,
the sandwich inequalities in \eqref{eq1: 20.03.26} and \eqref{eq2: 20.03.26} and the
lower bound on the surplus of a graph in Theorem~\ref{theorem: surplus LB} are particularly
useful since they provide bounds on graph invariants whose computation is, in general, NP-hard.

\section{Application}
\label{sec: applications}

While the algorithm for solving the SDP problem in \eqref{eq: SDP problem - Lovasz theta-function}
is polynomial in $\card{\V{\Gr{G}}}$, for the Gilbert graph $\Gilbert{q,n,d}$ we have
$\card{\V{\Gilbert{q,n,d}}} = q^n$. Hence, although the algorithm is polynomial in the number of vertices,
its complexity is exponential in $n$.
However, for any set of parameters $(q,n,d)$ such that $\Gilbert{q,n,d}$ or its complement is edge-transitive,
it is possible to compute the exact value of the Lov\'{a}sz $\vartheta$-function of both $\Gilbert{q,n,d}$ and its
complement $\overline{\Gilbert{q,n,d}}$.
All Gilbert graphs are vertex-transitive, since they are Cayley graphs. Thus, by
Theorem~\ref{thm: properties of the Lovasz function}(\ref{item: Lovasz product}), for every choice of parameters
$(q,n,d)$, we have $\vartheta(\Gilbert{q,n,d}) \cdot \vartheta(\overline{\Gilbert{q,n,d}}) = q^n$.

Combining the equality conditions in Theorem~\ref{thm:bounds on the Lovasz function for regular graphs} with the
complete classification of edge-transitive Gilbert graphs and their complements from Theorems~\ref{thm: main result}
and \ref{thm: main result complement}, we obtain exact closed-form expressions for the Lov\'{a}sz $\vartheta$-function
in the following cases, summarized in the next theorem.
\begin{theorem}
\label{thm: Lovasz theta function of Gilbert graphs}
Let $\Gr{G} = \Gilbert{q,n,d}$ be a Gilbert graph with regularity $\Delta \triangleq \sum_{i=1}^{d-1} \binom{n}{i} (q-1)^{i}$. Then
\begin{enumerate}
\item If either $d = 2$, or $(q,d) = (2,3)$, or $(q,d) = (2,n)$, then
\begin{align}
\vartheta(\Gr{G}) &= -\frac{q^n \, \Eigval{min}{\Gr{G}}}{\Delta - \Eigval{min}{\Gr{G}}} \\
\vartheta(\CGr{G}) &= 1 - \frac{\Delta}{\Eigval{min}{\Gr{G}}},
\end{align}
where $\Eigval{min}{\Gr{G}}$ is the smallest eigenvalue of the adjacency matrix of $\Gr{G}$.
\item If either $(q,d)=(2,n-1)$ for an even $n$, or $d = n$, then
\begin{align}
\vartheta(\Gr{G}) &= \frac{q^n - \Delta + \Eigval{2}{\Gr{G}}}{1 + \Eigval{2}{\Gr{G}}} \\
\vartheta(\CGr{G}) &= \frac{q^n (1 + \Eigval{2}{\Gr{G}})}{q^n - \Delta + \Eigval{2}{\Gr{G}}},
\end{align}
where $\Eigval{2}{\Gr{G}}$ is the second largest eigenvalue of the adjacency matrix of $\Gr{G}$.
\end{enumerate}
\end{theorem}
The smallest and second largest eigenvalues $\lambda_{min}$ and $\lambda_{2}$ are obtained from
equalities \eqref{eq: 20.03.26}. The complexity for computing
these eigenvalues is $O(n^2)$, which is a significant improvement over the general SDP-based approach,
whose complexity is exponential in $n$ for this family of graphs.

We next illustrate an application of knowing the exact value of the Lov\'{a}sz $\vartheta$-function,
as provided by Theorem~\ref{thm: Lovasz theta function of Gilbert graphs}.

\subsection*{The isomorphism problem}
To determine whether two graphs are non-isomorphic, one may compare various graph invariants computed for each graph.
Some invariants, such as the independence number or the chromatic number, are NP-hard to compute.
On the other hand, the spectra of matrix representations of a graph, such as the adjacency, Laplacian, signless Laplacian,
and normalized Laplacian matrices, can be computed in polynomial time; however, there exist cospectral non-isomorphic graphs.

Another useful invariant for distinguishing graphs is the Lov\'{a}sz $\vartheta$-function.
In contrast to spectral invariants, knowing the \emph{exact} values of the Lov\'{a}sz $\vartheta$-function (of a graph or its complement)
can be decisive for the isomorphism problem: if these values differ for two graphs, then the graphs are necessarily non-isomorphic.
In contrast, comparing only upper and lower bounds on the $\vartheta$-function may be inconclusive, as different bounds do not imply non-isomorphism.

This distinction is illustrated by concrete examples.
In \cite[Example~3]{SasonKHB2025}, a pair of non-isomorphic graphs is exhibited that share the same adjacency, Laplacian,
signless Laplacian, and normalized Laplacian spectra, and also have identical independence and clique numbers, yet have
distinct Lov\'{a}sz $\vartheta$-values.
More generally, \cite[Theorem~4.19]{Sason2024} shows that for every even integer $n \ge 14$, there exists a pair of connected,
irregular, non-isomorphic graphs on $n$ vertices that share all four spectra mentioned above, as well as identical independence,
clique, and chromatic numbers, but are distinguished by their Lov\'{a}sz $\vartheta$-function values.

\section*{Acknowledgements}
The authors are grateful to Ronny Roth, Raphael Yuster, and the anonymous reviewer for carefully reading an earlier version of 
the manuscript and providing valuable comments that improved its presentation.


\begin{thebibliography}{99}

\bibitem{AffifChaoucheB2006}
F. Affif Chaouche, A. Berrachedi:
Automorphisms group of generalized Hamming graphs.
\emph{Electron. Notes Discrete Math.} \textbf{24}, 9--15 (2006).
\url{https://doi.org/10.1016/j.endm.2006.06.003}

\bibitem{BallaJS2024}
I. Balla, O. Janzer, B. Sudakov:
On MaxCut and the Lov\'asz theta function.
\emph{Proc. Amer. Math. Soc.} \textbf{152}(5), 1871--1879 (2024).
\url{https://doi.org/10.1090/proc/16675}

\bibitem{BrouwerCN1989}
A. E. Brouwer, A. M. Cohen, A. Neumaier:
Theory of distance-regular graphs.
In: \emph{Distance-Regular Graphs},
Ergebnisse der Mathematik und ihrer Grenzgebiete, vol.~18,
pp.~126--166. Springer, Berlin/Heidelberg (1989).
\url{https://doi.org/10.1007/978-3-642-74341-2_4}

\bibitem{BrouwerHaemers2012}
A. E. Brouwer, W. H. Haemers:
\emph{Spectra of Graphs}.
Universitext. Springer, New York (2012).
\url{https://doi.org/10.1007/978-1-4614-1939-6}

\bibitem{CioabaM22}
S. M. Cioab\v{a}, M. R. Murty:
\emph{A First Course in Graph Theory and Combinatorics},
2nd ed.
Texts and Readings in Mathematics, vol.~55.
Springer, Singapore (2022).
\url{https://doi.org/10.1007/978-981-19-0957-3}

\bibitem{Coleman2011}
R. Coleman:
On Krawtchouk polynomials.
HAL preprint, hal-00554167 (2011).
\url{https://hal.science/hal-00554167}

\bibitem{Delsarte1973}
P. Delsarte:
\emph{An Algebraic Approach to the Association Schemes of Coding Theory}.
Philips Res. Rep. Suppl., no.~10, vi+97~pp. (1973).

\bibitem{DelsarteLevenshtein1998}
P. Delsarte, V. I. Levenshtein:
Association schemes and coding theory.
\emph{IEEE Trans. Inform. Theory} \textbf{44}(6), 2477--2504 (1998).
\url{https://doi.org/10.1109/18.720545}

\bibitem{Fritz21}
T. Fritz:
A unified construction of semiring-homomorphic graph invariants.
\emph{J. Algebraic Combin.} \textbf{54}(3), 693--718 (2021).
\url{https://doi.org/10.1007/s10801-020-00983-y}

\bibitem{GodsilRoyle2001}
C. Godsil, G. Royle:
\emph{Algebraic Graph Theory}.
Graduate Texts in Mathematics, vol.~207.
Springer, New York (2001).
\url{https://doi.org/10.1007/978-1-4613-0163-9}

\bibitem{JiangVardy2004}
T. Jiang, A. Vardy:
Asymptotic improvement of the Gilbert--Varshamov bound on the size of
binary codes.
\emph{IEEE Trans. Inform. Theory} \textbf{50}(8), 1655--1664 (2004).
\url{https://doi.org/10.1109/TIT.2004.831751}

\bibitem{Knuth94}
D. E. Knuth:
The sandwich theorem.
\emph{Electron. J. Combin.} \textbf{1}(1), Article A1, 48~pp. (1994).
\url{https://doi.org/10.37236/1193}

\bibitem{LidlNiederreiter96}
R. Lidl, H. Niederreiter:
\emph{Finite Fields}, 2nd ed.
Encyclopedia of Mathematics and its Applications, vol.~20.
Cambridge University Press, Cambridge (1996).
\url{https://doi.org/10.1017/CBO9780511525926}

\bibitem{LiZW2023}
Y. Li, J. Zhang, M. Wang:
The square of some generalized Hamming graphs.
\emph{Mathematics} \textbf{11}(11), Article 2487, 21~pp. (2023).
\url{https://doi.org/10.3390/math11112487}

\bibitem{Lovasz79}
L. Lov\'{a}sz:
On the Shannon capacity of a graph.
\emph{IEEE Trans. Inform. Theory} \textbf{25}(1), 1--7 (1979).
\url{https://doi.org/10.1109/TIT.1979.1055985}

\bibitem{MacWilliamsSloane77}
F. J. MacWilliams, N. J. A. Sloane:
\emph{The Theory of Error-Correcting Codes}.
North-Holland Mathematical Library, vol.~16.
North-Holland, Amsterdam (1977).
\url{https://doi.org/10.1016/S0924-6509(08)X7030-8}

\bibitem{McElieceRR1978}
R. J. McEliece, E. R. Rodemich, H. C. Rumsey, Jr.:
The Lov\'asz bound and some generalizations.
\emph{J. Combin. Inform. Syst. Sci.} \textbf{3}(3), 134--152 (1978).
\url{https://ipnpr.jpl.nasa.gov/progress_report2/42-45/45I.PDF}

\bibitem{Mirafzal2021}
S. M. Mirafzal:
Some remarks on the square graph of the hypercube.
arXiv preprint arXiv:2101.01615 (2021).
\url{https://doi.org/10.48550/arXiv.2101.01615}

\bibitem{Orel23}
M. Orel:
The core of a complementary prism.
\emph{J. Algebraic Combin.} \textbf{58}(3), 589--609 (2023).
\url{https://doi.org/10.1007/s10801-023-01236-4}

\bibitem{Polyanskiy2019}
Y. Polyanskiy:
Hypercontractivity of spherical averages in Hamming space.
\emph{SIAM J. Discrete Math.} \textbf{33}(2), 731--754 (2019).
\url{https://doi.org/10.1137/15M1046575}

\bibitem{QianWX2022}
C. Qian, Y. Wu, Y. Xiong:
Phylogeny numbers of generalized Hamming graphs.
\emph{Bull. Malays. Math. Sci. Soc.} \textbf{45}(5), 2733--2744 (2022).
\url{https://doi.org/10.1007/s40840-022-01338-5}

\bibitem{Roberson16}
D. E. Roberson:
Conic formulations of graph homomorphisms.
\emph{J. Algebraic Combin.} \textbf{43}(4), 877--913 (2016).
\url{https://doi.org/10.1007/s10801-016-0665-y}

\bibitem{Sason23}
I. Sason:
Observations on Lov\'asz $\vartheta$-function, graph capacity, eigenvalues,
and strong products.
\emph{Entropy} \textbf{25}(1), Article 104, 40~pp. (2023).
\url{https://doi.org/10.3390/e25010104}

\bibitem{Sason2024}
I. Sason:
Observations on graph invariants with the Lov\'asz $\vartheta$-function.
\emph{AIMS Math.} \textbf{9}(6), 15385--15468 (2024).
\url{https://doi.org/10.3934/math.2024747}

\bibitem{Sason2025}
I. Sason:
An example showing that Schrijver's $\vartheta$-function need not upper
bound the Shannon capacity of a graph.
\emph{AIMS Math.} \textbf{10}(7), 15294--15301 (2025).
\url{https://doi.org/10.3934/math.2025685}

\bibitem{SasonKHB2025}
I. Sason, N. Krupnik, S. Hamud, A. Berman:
On spectral graph determination.
\emph{Mathematics} \textbf{13}(4), Article 549 (2025).
\url{https://doi.org/10.3390/math13040549}

\bibitem{Tolhuizen1997}
L. M. Tolhuizen:
The generalized Gilbert--Varshamov bound is implied by Tur\'{a}n's theorem.
\emph{IEEE Trans. Inform. Theory} \textbf{43}(5), 1605--1606 (1997).
\url{https://doi.org/10.1109/18.623158}

\bibitem{WangCLL2026}
Y. Wang, M. Cao, Z. Lv, M. Lu:
Treewidth of generalized Hamming graph, bipartite Kneser graph and
generalized Petersen graph.
\emph{Electron. J. Combin.} \textbf{33}(1), Article P1.7 (2026).
\url{https://doi.org/10.37236/12892}

\bibitem{YeZLWYMa2023}
Z. Ye, H. Zhang, R. Li, J. Wang, G. Yan, Z. Ma:
Improving the Gilbert--Varshamov bound by graph spectral method.
\emph{CSIAM Trans. Appl. Math.} \textbf{4}(1), 1--12 (2023).
\url{https://doi.org/10.4208/csiam-am.SO-2021-0024}

\end{thebibliography}
\end{document}